\documentclass[12pt]{scrartcl}

\usepackage{url}

\usepackage{enumerate} 
\usepackage{mathrsfs} 
\usepackage{xcolor}

\usepackage{authblk}

\usepackage{amsmath,amssymb,amsthm} 

\author[1]{K. Dohmen} \affil[1]{Department of Mathematics\\ Mittweida
  University of Applied Sciences\\ Germany}

\author[2]{M. Trinks}
\affil[2]{Center for Combinatorics\\ Nankai University\\ Tianjin, China}

\title{An Abstraction of Whitney's Broken Circuit Theorem}

\newtheorem{theorem}{Theorem}

\theoremstyle{remark}
\newtheorem{remark}{Remark}
\DeclareMathOperator{\lcm}{lcm}
\DeclareMathOperator{\re}{Re}

\begin{document}

\maketitle

\begin{abstract}
We establish a broad generalization of Whitney's broken circuit theorem on the
chromatic polynomial of a graph to sums of the type $\sum_{A\subseteq S} f(A)$
where $S$ is a finite set and $f$ is a mapping from the power set of $S$ to an
abelian group. We give applications to the domination polynomial and the
subgraph component polynomial of a graph, the chromatic polynomial of a
hypergraph, the characteristic polynomial and Crapo's beta invariant of a
matroid, and the principle of inclusion-exclusion.  Thus, we discover several
known and new results in a concise and unified way.  As further applications
of our main result, we derive a new generalization of the maximums-minimums
identity and of a theorem due to Blass and Sagan on the M\"obius function of a
finite lattice, which generalizes Rota's crosscut theorem.  For the classical
M\"obius function, both Euler's totient function and its Dirichlet inverse,
and the reciprocal of the Riemann zeta function we obtain new expansions
involving the greatest common divisor resp.\ least common multiple.  We
finally establish an even broader generalization of Whitney's broken circuit
theorem in the context of convex geometries (antimatroids).

\footnotetext[1]{Electronic address:
  \url{dohmen@hs-mittweida.de}} \footnotetext[2]{Electronic address:
  \url{martin.trinks@googlemail.com}}
\end{abstract}

%%%%%%%%%%%%%%%%%%%%%%%%%%%%%%%%%%%%%%%%%%%%%%%%%%%%%%%
\section{Introduction}

Whitney's broken circuit theorem \cite{Whitney:1932:LEM} is one of the most
significant results on the chromatic polynomial of a graph.  We refer to
Diestel \cite{Diestel:2010:GT} for general graph terminology, and to Dong, Koh
and Teo \cite{DKT:2005:CPCG} for a comprehensive treatment of the chromatic
polynomial.  The chromatic polynomial of any finite simple graph $G=(V,E)$ can
be expressed as
\begin{gather}
\label{Birkhoff-Whitney}
P(G,x) = \sum_{A\subseteq E} (-1)^{|A|} x^{c(V,A)}, 
\end{gather}
where $c(V,A)$ denotes the number of connected components of the spanning
subgraph $(V,A)$.  The significance of the chromatic polynomial lies in the
fact that for any $x\in\mathbb{N}$ it evaluates to the number of proper
$x$-colourings of $G$, that is, the number of mappings
$f:V\rightarrow\{1,\dots,x\}$ such that $f(v)\neq f(w)$ for any edge
$\{v,w\}\in E$.  This interpretation matches the original definition due to
Birkhoff \cite{Birkhoff:1912:DFN}, whereas the expansion in
Eq.~(\ref{Birkhoff-Whitney}) goes back to Whitney \cite{Whitney:1932:LEM}.
In this paper, we adopt Eq.~(\ref{Birkhoff-Whitney}) as a definition.

In order to state Whitney's broken circuit theorem, we assume that the edge
set of $G$ is endowed with a linear ordering relation. Given a set $C$
consisting of the edges of a cycle of $G$, we refer to $C\setminus\{\max C\}$
as a \emph{broken circuit} of $G$.  Thus, a broken circuit of $G$ is obtained
from the edge set of a cycle of $G$ by removing its maximum edge.

In his prominent result, Whitney \cite{Whitney:1932:LEM} showed that the sum
in Eq.~(\ref{Birkhoff-Whitney}) can be restricted to those subsets $A$ which
do not include any broken circuit as a subset; that is,
\begin{gather}
P(G,x) = \sum_{\substack{A\subseteq E\\ \forall
  B\in\mathscr{B}:\, B\not\subseteq A}} (-1)^{|A|} x^{c(V,A)} 
\end{gather}
where $\mathscr{B}$ denotes the set of broken circuits of $G$.

As a consequence, since $c(V,A) = |V|-|A|$ whenever $(V,A)$ is cycle-free, the
coefficient of $x^{|V|-k}$ in $P(G,x)$ equals $(-1)^k$ times the number of
$k$-subsets of the edge set of $G$ which do not include any broken circuit of
$G$ as a subset ($k=0,1,2,\dots$).

The significance of Whitney's broken circuit theorem lies in the fact that it
provides a combinatorial interpretation of the coefficients of the chromatic
polynomial. It has been generalized to hypergraphs
\cite{Dohmen:1995:BCT,Trinks:2012:NBCT}, matroids\cite{Heron:1972:MP},
lattices \cite{BS:1997,Rota:1964}, generalized graph colourings
\cite{Dohmen:2003:NTV}, and sophisticated inclusion-exclusion variants
\cite{Dohmen:1999:IIE}.  In this paper, an even broader generalization is
established, from which the aforementioned generalizations derive in a concise
and unified way. Some new results are deduced as well, among them a broken
circuit theorem for the recent subgraph component polynomial
\cite{AMT:2010,TAM:2011}, a generalization of the Blass-Sagan theorem on the
M\"obius function of a finite lattice \cite{BS:1997}, and a generalization of
the well-known maximum-minimums identity \cite{Ross:2012:FCP}.

The paper is organized as follows. Section \ref{sec:main} contains the main
result along with two different proofs.  This main result generalizes
Whitney's broken circuit theorem to sums of the type $\sum_{A\subseteq S}
f(A)$ where $S$ is a finite set and $f$ is a mapping from the power set of $S$
to an abelian group.  In Section \ref{sec:applications} conclusions are
drawn for the chromatic polynomial of a hypergraph, the subgraph component
polynomial and the domination polynomial of a graph, the characteristic
polynomial and beta invariant of a matroid, the maximums-minimums-identity,
the principle of inclusion-exclusion, the M\"obius function of a lattice, the
classical M\"obius function, Euler's totient function and its Dirichlet
inverse, and the reciprocal of the Riemann zeta function.  In Section
\ref{sec:furth-gener} our main result is even further generalized to convex
geometries (a concept equivalent to antimatroids).  Roughly speaking, this
generalization states that, if its requirements are fulfilled, the sum
$\sum_{A\subseteq S} f(S)$ can be restricted to the free sets of a convex
geometry on $S$.

%%%%%%%%%%%%%%%%%%%%%%%%%%%%%%%%%%%%%%%%%%%%%%%%%%%%%%%
\section{Main result}
\label{sec:main}

Our main result, which is stated below, specializes to Whitney's broken
circuit theorem for any finite simple graph $G$ by letting $S$ be the edge set
of $G$, $\mathscr{C}$ the set of all edge sets of cycles of $G$,
$\Gamma=\mathbb{Z}[x]$ with the usual addition of polynomials, $f:2^S
\rightarrow \Gamma$ defined by $f(A) = (-1)^{|A|} x^{c(V,A)}$ for any
$A\subseteq S$, and $\mathscr{B} = \{ C\setminus\{\max C\} \mathrel|
C\in\mathscr{C}\}$.

% Main result
\begin{theorem}
\label{thm:1}
Let $S$ be a finite linearly ordered set, $\mathscr{C}\subseteq
2^S\setminus\{\emptyset\}$, $\Gamma$ an abelian group (additively written),
and $f:2^S\rightarrow \Gamma$ a mapping such that for any $C\in\mathscr{C}$
and $A\supseteq C$,
\begin{gather}
\label{Bedingung}
 f(A) + f(A\setminus \{\max C\}) = 0 \, .
\end{gather}
Then, for any $\mathscr{B}\subseteq \{ C\setminus\{\max C\} \mathrel|
C\in\mathscr{C} \}$,
\begin{gather}
\label{claim}
 \sum_{A \subseteq S} f(A) \,\, = \!\! \sum_{\substack{A\subseteq S \\ \forall
  B\in\mathscr{B}:\, B\not\subseteq A}} f(A) \, .
\end{gather} 
\end{theorem}

Subsequently, we give two proofs of Theorem \ref{thm:1}. The first proof makes
use of the principle of inclusion-exclusion, while the second proof is by
induction on $|\mathscr{B}|$.

% First proof
\begin{proof}[First Proof]
If $\emptyset\in\mathscr{B}$, then 
$C=\{c\}$ for some $C\in\mathscr{C}$ and $c\in S$, and hence,
by the requirement of the theorem,
$f(A) + f(A\setminus\{c\}) = 0$ for any $A\supseteq\{c\}$.
It follows that
\begin{gather*}
 \sum_{A\subseteq S} f(A) 
= \sum_{\substack{A\subseteq S\\ A\ni c}} f(A) + \sum_{\substack{A\subseteq
    S\\ A\not\ni c}} f(A)
= \sum_{\substack{A\subseteq S\\ A\ni c}} \left( f(A) + f(A\setminus\{c\})
\right) = 0 \, ,
\end{gather*}
which implies the validity of Eq.~(\ref{claim}) since no set $A$ satisfies
$\emptyset\not\subseteq A$.
\par
In the sequel, we assume that $\emptyset\notin \mathscr{B}$.  By the principle
of inclusion-exclusion,
\begin{align}
\sum_{\substack{A\subseteq S\\ \forall B\in\mathscr{B}:\,B\not\subseteq A}}\! f(A) 
& = \sum_{\mathscr{A}\subseteq\mathscr{B}} (-1)^{|\mathscr{A}|}
\sum_{\substack{A\subseteq S\\ A\supseteq\bigcup\mathscr{A}}} f(A) \notag \\
& = \sum_{A\subseteq S} f(A) + \bigg[
\sum_{\substack{\mathscr{A}\subseteq\mathscr{B} \\
    \mathscr{A}\neq\emptyset}} (-1)^{|\mathscr{A}|} \sum_{\substack{A\subseteq
S\\ A\supseteq \bigcup\mathscr{A}}} f(A) \bigg] \, . 
\label{bracketed}
\end{align} 
\newcommand{\X}{\bigcup \mathscr{A}} 
It remains to show that the bracketed term in Eq.~(\ref{bracketed}) 
vanishes, which is the case if
\begin{gather}
\label{Behauptung}
\sum_{A\supseteq \X}
f(A)=0 \qquad (\emptyset\neq \mathscr{A}\subseteq\mathscr{B}) .
\end{gather}
In order to establish Eq.~(\ref{Behauptung}), choose
(i) $B\in\mathscr{A}$ with $\max B = \max \X$, and 
(ii) $b\notin B$ such that $B\cup \{b\} \in \mathscr{C}$ and $b> \max B$.  
Then, $b\notin \X$ since otherwise $b\le \max \X = \max B$.  Hence,
\begin{gather}
\label{eq1}
 \sum_{A\supseteq \X} f(A) = \sum_{\substack{A\supseteq \X\\ A\ni b}} f(A) \,+ \sum_{\substack{A\supseteq \X\\
  A\not\ni b}} f(A) = \sum_{\substack{A\supseteq \X\\ A\not\ni b}} \left(
f(A\cup\{b\}) + f(A) \right) .
\end{gather}
For each $A$ in the last sum consider $A'=A\cup\{b\}$.  Since $A\supseteq
\X\supseteq B$ we have $A'\supseteq C$ for $C=B\cup\{b\}$ and hence, by the
requirement of the theorem,
\begin{gather}
\label{eq2}
 f(A\cup\{b\}) + f(A) = f(A') + f(A'\setminus\{ \max C\} ) = 0 \, .
\end{gather}
Now Eqs.~(\ref{eq1}) and (\ref{eq2}) imply Eq.~(\ref{Behauptung}), and hence
the statement of the theorem.
\end{proof}

% Second proof
\begin{proof}[Second Proof]
The statement is obvious if $\mathscr{B}=\emptyset$.  We proceed by induction
on $|\mathscr{B}|$.  If $\mathscr{B}\neq\emptyset$, then for some
$C\in \mathscr{C}$, $C\setminus \{\max C\}\in \mathscr{B}$.  Among those $C$
choose one whose $\max C$ value is maximal.  Let
$\mathscr{B}':=\mathscr{B}\setminus \{C\setminus \{\max C\} \}$.  By the
induction hypothesis,
\begin{gather*}
 \sum_{A \subseteq S} f(A) \,\, = \!\! \sum_{\substack{A\subseteq S \\ \forall
  B\in\mathscr{B}':\, B\not\subseteq A}} f(A) \, ,
\end{gather*} 
which implies
\begin{gather}
\label{afterinduction}
 \sum_{A \subseteq S} f(A) \,\, = \!\! \sum_{\substack{A\subseteq S \\ \forall
  B\in\mathscr{B}:\, B\not\subseteq A}} f(A) \,+
 \sum_{\substack{A\subseteq S \\ \forall
  B\in\mathscr{B}':\, B\not\subseteq A\\ C\setminus\{\max C\}\subseteq A}} f(A).
\end{gather}
We claim that the second sum on the right-hand side 
of Eq.~(\ref{afterinduction}) vanishes.
Let
\begin{gather*}
\mathscr{A} := \left\{ A \mathrel| A\subseteq S; \,\,\forall B\in\mathscr{B}': B\not\subseteq A; \,\,C\setminus\{\max C\} \subseteq A \right\} . %|
\end{gather*}
By Eq.~(\ref{Bedingung}) our claim is proved if $A\in\mathscr{A}$ if and only
if $A\setminus\{\max C\}\in\mathscr{A}$.  The only non-trivial issue is to
show that if $A\setminus\{\max C\}\in\mathscr{A}$, then $B\not\subseteq A$ for
any $B\in\mathscr{B}'$.  Assume that $B\subseteq A$ for some
$B\in\mathscr{B}'$.  By the requirement of the theorem there exists
$C'\in\mathscr{C}$ such that $B=C'\setminus\{\max C'\}$.  If $\max C \in B$,
then $\max C \le \max B < \max C'$, contradicting the maximality of $\max C$.
If $\max C\notin B$, then since $A\supseteq B$ we conclude that $A\setminus
\{\max C\}\supseteq B$, which is in contradiction with $A\setminus\{\max C\}
\in\mathscr{A}$.
\end{proof}

\begin{remark}
\label{thm:remark}
Let $S$ be a finite poset, $\mathscr{C} = \{\{s,t\} \subseteq S\mathrel|
s<t\}$, and $f$  a mapping satisfying the requirement in
Eq.~\eqref{Bedingung} with respect to some linear extension of $S$.  By
Theorem \ref{thm:1},
\begin{gather}
\label{eq:10}
\sum_{A\subseteq S} f(A) = \sum_{A\subseteq S_{\max}} f(A) 
\end{gather}
where $S_{\max}$ denotes the set of maximal elements in $S$.
\end{remark}

\begin{remark}
If $S$ is an upper semilattice, we may choose $\mathscr{C}=\{\{s,t,s\vee
t\}\mathrel| s||t\}$, where $s||t$ signifies that $s$ and $t$ are incomparable
and $s\vee t$ denotes the least upper bound of $s$ and $t$ in $S$. Thus,
for any mapping $f$ satisfying the requirement in Eq.~\eqref{Bedingung}
we have
\begin{gather}
\label{eq:11}
\sum_{A\subseteq S} f(A) = \sum_{\substack{A\subseteq S\\ \text{$A$
      chain}}} f(A) .
\end{gather}
\end{remark}

We will make use of the preceding two identities in Subsection
\ref{sec:arithmetic-functions}.

%%%%%%%%%%%%%%%%%%%%%%%%%%%%%%%%%%%%%%%%%%%%%%%%%%%%%%%
\section{Applications}
\label{sec:applications}

\subsection{Chromatic polynomial of a hypergraph}

A \emph{hypergraph} is a pair $H=(V,\mathscr{E})$ where $V$ is a set (of
\emph{vertices}) and $\mathscr{E}$ is a set of non-empty subsets of $V$
(called \emph{edges}).  $H$ is called \emph{finite} if $V$ is finite, and
\emph{simple} if $|E|\ge 2$ for any $E\in\mathscr{E}$.  Distinct vertices
$v,w\in V$ such that $v,w\in E$ for some $E\in\mathscr{E}$ are called
\emph{adjacent}. The reflexive and transitive closure of the adjacency
relation yields an equivalence relation on $V$, whose equivalence classes are
referred to as \emph{connected components} of $H$, and whose number of
equivalence classes is denoted by $c(H)$.

By applying the principle of inclusion-exclusion it follows that for any finite
simple hypergraph $H$ and any $x\in\mathbb{N}$ the polynomial
\begin{gather*}
P(H,x) = \sum_{A\subseteq \mathscr{E}} (-1)^{|A|} x^{c(V,A)}
\end{gather*}
evaluates to the number of mappings $f:V\rightarrow\{1,\dots,x\}$ such that
$f|_E$ (the restriction of $f$ to $E$) is non-constant for any $E\in
\mathscr{E}$ (see \cite{Dohmen:1995:BCT} for details).  

We consider cycles in hypergraphs in the classical sense of C. Berge
\cite{Berge:1973:GH}.  Accordingly, a \emph{cycle of length $l$} in a
hypergraph $H$ is any finite sequence
$(v_1,E_1,v_2,E_2,\dots,v_l,E_l,v_{l+1})$ consisting of at least two pairwise
distinct vertices $v_1,\dots,v_l\in V$ resp.\ edges $E_1,\dots,E_l\in
\mathscr{E}$ where $v_1 = v_{l+1}$ and $v_i,v_{i+1}\in E_i$ for
$i=1,\dots,l$. The definition of a broken circuit is similar as for graphs:
Given a linear ordering relation on $\mathscr{E}$, for any set $C$ consisting
of the edges of a cycle of $H$ we refer to $C\setminus\{\max C\}$ as a
\emph{broken circuit} of~$H$.

Let $H=(V,\mathscr{E})$ be a finite simple hypergraph whose edge set is
endowed with a linear ordering relation.  In order to apply Theorem
\ref{thm:1} we choose $S=\mathscr{E}$, $\mathscr{C}$ as a set of edge sets of
cycles of $H$, $\Gamma=\mathbb{Z}[x]$ with the usual addition of polynomials,
$f(A) = (-1)^{|A|} x^{c(V,A)}$ for any subsets $A\subseteq S$, and $\mathscr{B}
\subseteq \{C\setminus\{\max C\} \mathrel| C\in\mathscr{C} \} $.  This gives
\begin{gather}
\label{eq:7}
P(H,x) = \sum_{\substack{A\subseteq \mathscr{E}\\ \forall
  B\in\mathscr{B}:\, B\not\subseteq A}} (-1)^{|A|} x^{c(V,A)} 
\end{gather}
provided, of course, that the requirement in Eq.~(\ref{Bedingung}) is
satisfied for any $C\in\mathscr{C}$ and any $A\supseteq C$.  This can be
guaranteed by imposing one of the following requirements on $\mathscr{C}$:
\begin{enumerate}
\item[(a)] All cycles belonging to $\mathscr{C}$ have the property that
each edge of the cycle is included by the union of the other edges of that
cycle. 
\item[(b)] All cycles belonging to $\mathscr{C}$ contain an edge of
cardinality 2, and these edges constitute an upset of the edge
set with respect to the given linear ordering relation.
\end{enumerate}
When applied to (a), Theorem \ref{thm:1} from Section \ref{sec:main} provides
us with a new proof of \cite[Theorem 8]{Trinks:2012:NBCT}.  When applied to
(b), it leads to a new proof of \cite[Theorem 2]{Dohmen:1995:BCT}.

The requirement in (a) is satisfied if $\mathscr{C}$ arises from cycles in $H$
having the property that each edge on the cycle is included by the union of
its two neighbouring edges.  This latter condition holds, e.g., for
\emph{$l$-tight} cycles in $r$-uniform hypergraphs where $l\ge r/2$; these are
cycles ``whose vertices can be cyclically ordered in such a way that the edges
are segments of this ordering and every two consecutive edges intersect in
exactly $l$ vertices'' \cite{Glebov:2012:EHHC}.  Recall that a hypergraph is
referred to as \emph{$r$-uniform} if each edge contains exactly $r$
vertices. An $(r-1)$-tight cycle in an $r$-uniform hypergraph is called
\emph{tight}. Thus, choosing $\mathscr{C}$ from the tight cycles of an
$r$-uniform hypergraph satisfies (a).

As a more concrete example for (a), consider the 4-uniform hypergraph
$H=(V,\mathscr{E})$ on the set of lattice points of a finite rectangular grid
where the edges of $H$ are any four points determining a rectangle. Let
$\mathscr{C}$ be the set of all 3-sets of edges arising from the 2-tight
cycles of length three in $H$.  Since any 2-tight cycle of length three
corresponds to a pair of neighbouring rectangles (that is, rectangles having
two points in common, thus determining another, geometrically larger
rectangle) we can order the edges of $H$ in such a way that edges
corresponding to geometrically larger rectangles occur later in the ordering.
In this way, the sum in Eq.~\eqref{eq:7} can be restricted to those subsets
$A$ of $\mathscr{E}$ that contain no neighbouring rectangles.  Rectangle-free
grid colorings are a topic of active research; see e.g., \cite{Steinbach:2014}
for recent results.

\subsection{Subgraph component polynomial}
\label{sec:subgr-comp-polyn}

Introduced by Averbouch, Makowsky, and Tittmann \cite{AMT:2010,TAM:2011}, the
\emph{subgraph component polynomial} of any finite graph $G=(V,E)$ is defined
by
\begin{gather*} 
Q(G,x,y) = \sum_{A\subseteq V} x^{|A|} y^{c(G[A])}.
\end{gather*}
This polynomial has seen applications in social network analysis
\cite{AMT:2010} and formal language theory \cite{BBFH:2013}.  For some recent
results on $Q(G,x,y)$, the reader is referred to \cite{LH:2013:NSCP}.

In the following, our considerations are restricted to the particular case
where $x=-1$.  We refer to $G$ as \emph{cyclically claw-free} if no centre of
a claw is located on a cycle.  Evidently, any claw-free or cycle-free graph is
cyclically claw-free.

The key observation is that if $G$ is cyclically claw-free, then $c(G[A]) =
c(G[A\setminus \{c\}])$ for any $A\subseteq V$ and any vertex $c$ on a cycle
$C\subseteq A$ (where, in this subsection, we consider cycles as subsets of
the vertex set). This leads to a vertex analogue of the notion of a broken
circuit: Given a linear ordering relation on $V$, for any cycle $C\subseteq V$
we refer to $C\setminus \{\max C\}$ as a \emph{broken circuit} of $G$. Similar
to Eq.~\eqref{eq:7} we obtain by Theorem \ref{thm:1},
\begin{gather}
\label{eq:5}
Q(G,-1,y) = \sum_{\substack{A\subseteq V\\ \forall B\in\mathscr{B}: B\not\subseteq A}}
(-1)^{|A|} y^{c(G[A])}
\end{gather}
for any cyclically claw-free finite graph $G$ and any set
$\mathscr{B}\subseteq 2^V$ of broken circuits of $G$.  Note that if we choose
$\mathscr{B}$ as the set of \emph{all} broken circuits of $G$, then any subset
$A$ of $V$ in the preceding sum is cycle-free, and hence satisfies $c(G[A]) =
|A| - m(G[A])$. Thus, we obtain
\begin{gather}
\label{eq:6}
Q(G,-1,y) = \sum_{A} (-1)^{|A|} y^{|A|-m(G[A])}
\end{gather}
where the sum extends over all subsets $A$ of $V$ not including any broken
circuit. As a consequence, $Q(G,-1,-1)$ is the number of broken-circuit-free
vertex-induced subgraphs having an even number of edges minus those having an
odd number of edges.

We finally remark that Eqs.~\eqref{eq:5} and \eqref{eq:6} hold with $x$ in
place of $-1$ if $Q(G,x,y)$ is considered as a polynomial over some
commutative ring where $(x+1)y=0$.

\subsection{Domination polynomial}

The domination polynomial of any finite simple graph $G=(V,E)$, introduced by
Arocha and Llano \cite{Arocha-Llano:2000:MVM}, is the generating function
\begin{gather*}
D(G,x) := \sum_{k=0}^{|V|} d_k(G) x^k 
\end{gather*}
where $d_k(G)$ is the number of $k$-subsets $A$ of $V$ satisfying
$N_G[A]=V$. Here, $N_G[A]$ denotes the closed neighbourhood of $A$ in $G$,
that is, the union of $A$ and its set of neighbours in $G$.  For convenience,
we write $N_G[v]$ in place of $N_G[\{v\}]$ for any $v\in V$.

Given a linear ordering relation on the vertex set of $G$, for any $v\in
V$ we refer to $N_G[v]\setminus \{v\}$ as a \emph{broken neighbourhood} of $G$
if $ v = \max N_G[v]$.  In \cite{DT:2012:DR} it is shown that
\begin{gather}
\label{bnh}
 D(G,x) = \sum_{A\subseteq V} (-1)^{|A|} (x+1)^{|V|-|N_G[A]|} ,
\end{gather}
and moreover, \emph{if $G$ does not have isolated vertices, then this sum can be
  restricted to those subsets $A$ of $V$ which do not include any broken
  neighbourhoods from an arbitrary set of broken neighbourhoods of $G$.}
\par
This latter statement easily derives from our main result in Section
\ref{sec:main} and Eq.~(\ref{bnh}) by considering the mapping $f(A) =
(-1)^{|A|} (x+1)^{|V|-|N_G[A]|}$ for any $A\subseteq V$ and letting
$\mathscr{C}$ be the set of all closed neighbourhoods $N_G[v]$ where $v=\max
N_G[v]$.  The requirement in Eq.~(\ref{Bedingung}) is satisfied since
$N_G[A\setminus\{v\}] = N_G[A]$ for any $v\in V$ and any $A\supseteq N_G[v]$.

As noted in \cite{DT:2012:DR}, if $G$ does not have isolated vertices or
isolated edges, and its vertex set is linearly ordered such that the vertices
of degree 1 constitute an upset, then each pendant edge $\{v,w\}$ where $v$ is
of degree 1 gives rise to a broken neighbourhood $\{w\}$.  In this case, the
sum in Eq.~(\ref{bnh}) can be restricted to those subsets $A$ of $V$ which do
not contain any vertex from a set of vertices which are adjacent to a vertex
of degree 1.

\subsection{Characteristic polynomial and beta invariant}

Similar conclusions as for graph and hypergraph polynomials can be drawn for
the characteristic polynomial \cite{Heron:1972:MP} and the beta invariant
\cite{Crapo:1967:HIM} of a matroid.

Recall that a \emph{matroid} is a pair $M=(E,r)$ consisting of a finite set
$E$ and a $\mathbb{Z}$-valued function $r$ on $2^E$ such that
for any $A,B\subseteq E$, 
\begin{enumerate}[(i)]
\item $0\le r(A) \le |A|$, 
\item $A\subseteq B \Rightarrow r(A) \le r(B)$, 
\item $r(A\cup B) + r(A\cap B) \le r(A) + r(B)$.  
\end{enumerate}
A \emph{circuit} of $M$ is a non-empty subset $C\subseteq E$ such that
$r(C\setminus\{c\}) = |C|-1=r(C)$ for any $c\in C$.  Given a linear ordering
relation on $E$, for any circuit $C$ of $M$ we refer to $C\setminus\{\max C\}$
as a \emph{broken circuit} of~$M$.

The characteristic polynomial $\chi(M,x)$ and the beta invariant
$\beta(M)$ of a matroid $M=(E,r)$ are defined by
\begin{gather}
\chi(M,x) = \sum_{A\subseteq E} (-1)^{|A|} x^{r(E)-r(A)} \,
, \label{firstsum} \\
\beta(M) = (-1)^{r(E)} \sum_{A\subseteq E} (-1)^{|A|} r(A) \, . \label{secondsum}
\end{gather}
In order to apply Theorem \ref{thm:1}, let $f_1:2^E\rightarrow \mathbb{Z}[x]$
be defined by $f_1(A) = (-1)^{|A|} x^{r(E)-r(A)}$, and $f_2:2^E\rightarrow
\mathbb{Z}$ by $f_2(A) = (-1)^{|A|} r(A)$. Let $\mathscr{C}$ denote the set of
all circuits of $M$.  Then, for any $C\in\mathscr{C}$ and any $A\supseteq C$,
$r(A\setminus \{\max C\}) = r(A)$; hence, both $f_1$ and $f_2$ satisfy the
requirement in Eq.~(\ref{Bedingung}).  Let $\mathscr{B}$ denote the set of
broken circuits of $M$.  Then, by Theorem \ref{thm:1}, the sums in
Eqs.~(\ref{firstsum}) and (\ref{secondsum}) can be restricted to those subsets
$A$ of $E$ not including any $B\in\mathscr{B}$ as a subset.  No such $A$ may
include a circuit, since otherwise it would include the broken circuit derived
from it. Therefore, $r(A)=|A|$ and hence,
\begin{gather}
\chi (M,x) = \sum_{\substack{A\subseteq E\\ \forall
    B\in\mathscr{B}: B\not\subseteq A}} (-1)^{|A|} x^{r(E)-r(A)} 
                 \,=\,\sum_{k=0}^{|E|} (-1)^k b_k(M) x^{r(E)-k} \, , \label{Heron} \\
\beta (M)  = (-1)^{r(E)} \!\! \sum_{\substack{A\subseteq E\\ \forall
    B\in\mathscr{B}: B\not\subseteq A}} \!\! (-1)^{|A|} r(A) \,=\, (-1)^{r(E)} \sum_{k=1}^{|E|} (-1)^k k \, b_k(M) \label{NotHeron}
\end{gather}
where $b_k(M)$ denotes the number of $k$-subsets of $E$ including no broken
circuit.

Eq.~(\ref{NotHeron}) can alternatively be deduced from Eq.~(\ref{Heron}),
which is due to Heron \cite{Heron:1972:MP}, by considering the derivative of
$\chi(M,x)$ at $x=1$.

\subsection{Maximum-minimums identity}

Let $\Gamma$ be an abelian group, endowed with a linear ordering relation,
$(x_s|s\in S)$ a finite family of elements from $\Gamma$, $k\in\mathbb{N}$,
and $f:2^S\rightarrow \Gamma$ defined by
\begin{gather*}
 f(A) = \begin{cases} (-1)^{|A|-k} \min_k (x_a|a\in A) \, , & \text{if $|A|\ge k$},\\
                        0\, ,                 & \text{if $|A|<k$},
          \end{cases} 
\end{gather*}
where $\min_k (x_a|a\in A)$ denotes the $k$-th smallest element in $(x_a|a\in
A)$ for any $A\subseteq S$ satisfying $|A|\ge k$.

In order to define $\mathscr{B}$ and $\mathscr{C}$, choose some linear
ordering relation on $S$ such that $s<t$ implies $x_s\le x_t$ for any $s,t\in
S$.  Now, define $\mathscr{B} = \{ C\setminus \{\max C\} \mathrel|
C\in\mathscr{C} \}$ where $\mathscr{C}$ is the set of all ($k+1$)-subsets of
$S$. Evidently, for any $C\in\mathscr{C}$ and any $A\supseteq C$,
\begin{gather*}
\min\nolimits_k (x_a|a\in A\setminus\{\max C\}) = \min\nolimits_k (x_a|a\in
A) .
\end{gather*}
Hence, the requirements of Theorem \ref{thm:1} are satisfied, which gives
\begin{gather}
\label{eq:8}
 \sum_{\substack{A\subseteq S\\ |A|\ge k}} (-1)^{|A|-k} \min\nolimits_k
 (x_a|a\in A)  
= \! \sum_{\substack{A\subseteq S\\ |A|\ge k \\ \forall B\in\mathscr{B}: \,
    B\not\subseteq A}} \!  (-1)^{|A|-k} \min\nolimits_k (x_a|a\in A)   \, .
\end{gather}
Since $\mathscr{B}$ consists of all $k$-subsets of $S\setminus\{\max S\}$, the
last two conditions under the second sum in Eq.~\eqref{eq:8} are equivalent to
$|A|=k$ and $\max S\in A$.  Since there are ${|S|-1 \choose k-1}$ many such
$A$, and each of them satisfies $\min_k (x_a|a\in A) = \max (x_s | s\in S)$,
we find that
\begin{gather}
\label{eq:9}
 \sum_{\substack{A\subseteq S\\ |A|\ge k}} (-1)^{|A|-k} \min\nolimits_k
 (x_a|a\in A)
= {|S|-1 \choose k-1} \max (x_s|s\in S).
\end{gather}
Note that neither side of this identity depends on the ordering of $S$.  For
$k=1$ this identity is known as the \emph{maximum-minimums identity}; see
\cite{Ross:2012:FCP} for a probabilistic proof (in case that the $x_s$'s are
reals) and an application to the coupon collector problem.

\subsection{Principle of inclusion-exclusion}

Let $\{M_s\}_{s\in S}$ be a finite family of finite sets, where $S$ is
linearly ordered, $\mathscr{B}$ a set of non-empty subsets of $S$ such that
for any $B\in\mathscr{B}$, $\bigcap_{b\in B} M_b \subseteq M_c$ for some $c =
c(B) > \max B$.  In \cite{Dohmen:1999:IIE} it is shown that
\begin{gather}
\label{eq:4}
\left| \bigcup_{s\in S} M_s \right| \,=\, \sum_{\substack{\emptyset \neq A
    \subseteq S \\ \forall B\in\mathscr{B}: \,B\not\subseteq A}}
(-1)^{|A|-1} \left| \bigcap_{a\in A} M_a \right|. 
\end{gather}
Under the above assumptions, this identity (which has applications to network
and system reliability) follows from Theorem \ref{thm:1} by defining 
$\mathscr{C} = \{ B\cup c(B) \mathrel| B\in \mathscr{B}\}$
and 
\begin{gather*}
f(A) =
\begin{cases}
(-1)^{|A|-1} \left| \bigcap_{a\in A} M_a\right| , & \text{if
  $A\neq\emptyset$},\\
  0\, , & \text{if $A=\emptyset$},
\end{cases}
\end{gather*}
and applying the principle of inclusion-exclusion to the sets $M_s$,
$s\in S$.

A particular case of Eq.~\eqref{eq:4} is Narushima's principle of
inclusion-exclusion \cite{Nar:1974} where the sum extends over all chains of
a semilattice:
\begin{gather}
\label{eq:22}
\left| \bigcup_{s\in S} M_s \right| \,=\, \sum_{\substack{\emptyset \neq A
    \subseteq S \\ \text{$A$ chain}}}
(-1)^{|A|-1} \left| \bigcap_{a\in A} M_a \right|
\end{gather}
As a prerequisite, $(S,\vee)$ is required to be a finite upper semilattice
satisfying $M_s \cap M_t \subseteq M_{s\vee t}$ for any $s,t\in S$.  This
particular case of Eq.~\eqref{eq:4} may also be deduced from
Eq.~\eqref{eq:11}.

\subsection{M\"obius function of a lattice}
\label{sec:mobi-funct-latt}

Our next application concerns the M\"obius function of a finite lattice.  For
notions from the theory of partially ordered sets and lattices, we refer to
the textbook of Graetzer~\cite{Graetzer:1998:GLT}.

Recall that the \emph{M\"obius function} of any finite lattice $L = [\hat{0},
\hat{1}]$ is the unique $\mathbb{Z}$-valued function $\mu_L : L \rightarrow
\mathbb{Z}$ such that for any $x\in L$,
\begin{gather*}
\sum_{y\le x} \mu_L(y) = \delta_{\hat{0}x}
\end{gather*}
 where $\delta$ denotes the
Kronecker delta.  Following Rota \cite{Rota:1964}, we write $\mu(L)$ instead
of $\mu_L(\hat{1})$ and introduce the notion of a \emph{crosscut}, which is
any antichain $C\subseteq L\setminus \{\hat{0}, \hat{1}\}$ having a non-empty
intersection with any maximal chain from $\hat{0}$ to $\hat{1}$ in $L$.  As a
prerequisite, $L$ must be non-trivial, that is, $L\setminus \{\hat{0},
\hat{1}\}\neq\emptyset$.  Rota's crosscut theorem \cite{Rota:1964} states that for any
non-trivial finite lattice $L=[\hat{0},\hat{1}]$ and any crosscut $C$ of $L$,
\begin{gather}
\label{crosscut}
\mu(L) = \sum_{\substack{A\subseteq C\\ \bigwedge\! A=\hat{0}, \bigvee
    \! A=\hat{1}}} (-1)^{|A|}
\end{gather}
where $\bigwedge\emptyset = \hat{1}$ and $\bigvee \emptyset = \hat{0}$.
Due to Blass and Sagan \cite{BS:1997}, for $C=A(L)$, which is the crosscut of
all atoms of $L$, this sum can be written as%
\footnote{For $C=A(L)$ the conditions $\bigwedge A=\hat{0}$ in Eq.~(\ref{nbb})
and $\bigwedge B < c$ in Eq.~(\ref{condb}) can be omitted.}
\begin{gather}
\label{nbb}
\mu(L) = \sum_{\substack{A\subseteq C\\ \bigwedge\! A=\hat{0}, \bigvee
    \! A=\hat{1} \\ \forall B\in\mathscr{B}: \,B\not\subseteq A}} (-1)^{|A|}
\end{gather}
where, according to some fixed partial ordering relation $\trianglelefteq$ on
$C$, $\mathscr{B}$ consists of all non-empty subsets $B$ of $C$ such
that for any $b\in B$ there is some $c=c(B,b)\in C$ satisfying
\begin{gather}
\label{condb}
c \vartriangleleft b 
\quad \text{and}\quad
\bigwedge\! B < c < \bigvee\! B .
\end{gather}
Here and subsequently, $<$, $\wedge$ and $\vee$ are associated with the
lattice ordering $\le$ in $L$, while $\triangleleft$ is associated with the
additional partial ordering relation $\trianglelefteq$ on $C$. 

Blass and Sagan \cite{BS:1997} used their result in computing and
combinatorially explaining the M\"obius function of various lattices and in
generalizing Stanley's well-known theorem \cite{Stanley:1972} that the
characteristic polynomial of a semimodular supersolvable lattice factors over
the integers.  As noted by Blass and Sagan \cite{BS:1997}, for $C=A(L)$
Eq.~(\ref{nbb}) generalizes Eq.~(\ref{crosscut}), which is easily seen by
considering the total incomparability $\trianglelefteq$ order on $C$.
% In fact, Blass and Sagan
% \cite{BS:1997} did not employ Rota's crosscut theorem \cite{Rota:1964} in the
% proof of their result.

We now prove that Eq.~(\ref{nbb}) holds for \emph{any} crosscut $C$ of~$L$ by
applying our main result from Section \ref{sec:main} in dual form to the sum
in Eq.~(\ref{crosscut}).  Thus, we consider $f:2^C \rightarrow \mathbb{Z}$
where
\begin{align*}
f(A) & := \begin{cases} (-1)^{|A|} & \text{if $\bigwedge \! A = \hat{0}$ and
  $\bigvee \! A = \hat{1}$},\\
                        0 & \text{otherwise},
          \end{cases} 
\end{align*}
for any $A\subseteq C$.
According to some arbitrary linear extension of $\trianglelefteq$ on $C$ define
\begin{gather*}
 \mathscr{C} := \Big\{B\cup \big\{\!\min_{b\in B} c(B,b)\big\} \,\Big|\,
B\in\mathscr{B}\Big\}\, , 
\end{gather*}
which implies $\mathscr{B} = \{C^\prime\setminus \{\min C^\prime\}\mathrel|
C^\prime\in\mathscr{C}\}$.
It remains to check that for any
$C^\prime\in\mathscr{C}$ and any $A\supseteq C^\prime$ the requirement in
Eq.~(\ref{Bedingung}) holds.
To this end, we show that 
\begin{align}
\label{requirement}
\bigwedge A \,=\, \bigwedge \big( A\setminus \{\min C^\prime\} \big),
\quad 
\bigvee \! A \,=\, \bigvee \big( A\setminus \{\min C^\prime\} \big).
\end{align}
For the first identity in (\ref{requirement}), choose $B\in\mathscr{B}$ such
that $C^\prime = B \cup \{ \min_{b\in B} c(B,b) \}$. Then,
\begin{align*}
\bigwedge\! A & = \bigwedge (A\setminus C^\prime) \wedge \bigwedge C^\prime
\qquad \text{\textcolor{darkgray}{(since $A\supseteq C^\prime$)}}\\
& = \bigwedge \Big( A \setminus \big( B \cup \big\{ \! \min_{b\in B} c(B,b)
\big\} \big) \Big) \wedge \, \bigwedge \Big( B \cup \big\{ \! \min_{b\in B}
c(B,b) \big\}
\Big) \\
& = \bigwedge \Big( A \setminus \big( B \cup \big\{ \! \min_{b\in B} c(B,b)
\big\} \big) \Big) \wedge \, \bigwedge \! B \qquad \text{\textcolor{darkgray}{(since
  $\min_{b\in B} c(B,b)
  > \bigwedge B$)}} \\
& = \bigwedge \Big( A \setminus \big\{ \! \min_{b\in B} c(B,b) \big\} \Big)
\qquad \text{\textcolor{darkgray}{(since $A\supseteq B$ and $\min_{b\in B} c(B,b)\notin B$)}} \\
& = \bigwedge \Big( A \setminus \big\{ \! \min C^\prime \big\} \Big)
\qquad \text{\textcolor{darkgray}{(since $\min C^\prime = \min_{b\in B} c(B,b)$).}}
\end{align*}
For the second claim in Eq.~(\ref{requirement}), simply exchange $\wedge$
with $\vee$ and $>$ with $<$. Thus, for \emph{any} crosscut $C$ of~$L$, the
identity in Eq.~(\ref{nbb}) follows from our main result in Section
\ref{sec:main}.  Furthermore, the identity remains valid if $\mathscr{B}$ is
replaced by any subset $\mathscr{B}'\subseteq \mathscr{B}$.

% As an example, let $L$ be the lattice of positive divisors of $180 = 2^2\cdot
% 3^2\cdot 5^1$, and $C = \{4,6,9,10,15\}$.  If we define $6\vartriangleleft c$
% for any $c\in C\setminus\{6\}$, then $\mathscr{B} = \{ \{4,9\}, \{4,15\},
% \{9,10\} \}$. In this case, it turns out that the sum in Eq.~(\ref{nbb}) is
% empty and hence, $\mu(L)=0$. This is what the attentive reader might have
% expected since 180 is not squarefree.

\subsection{Arithmetical functions}
\label{sec:arithmetic-functions}

In this subsection, we establish new gcd- and lcm-sum expansions for some
classical arithmetical functions. We refer to the textbook of Apostol
\cite{Apostol:1976} for a comprehensive account of arithmetical functions in
general, and of multiplicative functions in particular.

\subsubsection{The classical M\"obius function}
\label{sec:class-mobi-funct}

For any $n\in\mathbb{N}$ let $L_n$ denote the lattice of positive divisors of
$n$, and $\mu(n)=\mu(L_n)$ the classical M\"obius function of $n$.  It is
well-known that for $n\ge 1$ and $k\ge 0$,
\begin{align}
\label{eq:18}
\mu(n) & = \begin{cases} (-1)^k, & \text{if $n$ is the product of $k$
  distinct primes,} \\  0,       & \text{otherwise}.
           \end{cases}
\end{align}                
We use $\gcd(A)$ and $\lcm(A)$ to denote the greatest common divisor
resp.\ least common multiple of any finite set $A\subseteq \mathbb{N}$.  We
adopt the convention that $\gcd(\emptyset)=0$ and \mbox{$\lcm(\emptyset)=1$}.
With $f(A)=(-1)^{|A|}$ if $\gcd(A)=1$ and $f(A)=0$ if $\gcd(A)>1$ for
$A\subseteq L_n\setminus\{1,n\}$ the first remark after Theorem \ref{thm:1}
implies that for any non-prime $n>1$,
\begin{gather}
\label{eq:12}
\sum_{\substack{A\subseteq L_n\setminus\{1,n\} \\ \gcd(A)=1}} (-1)^{|A|} =
\sum_{\substack{A\subseteq P_n^\ast \\ \gcd(A)=1}} (-1)^{|A|} =
\sum_{\substack{A\subseteq P_n\\ \gcd(A^\ast)=1}} (-1)^{|A|} =
\sum_{\substack{A\subseteq P_n \\ \lcm(A)=n}} (-1)^{|A|} = \mu(n)
\end{gather}
where $P_n$ denotes the set of prime factors of $n$, and $A^\ast =
\left\{\frac{n}{a}\mathrel| a\in A\right\}$ for any $A\subseteq P_n$.  The first
equality in Eq.~\eqref{eq:12} is due to Eq.~\eqref{eq:10}, while the last one
follows from Eq.~\eqref{eq:18} (or Eq.~\eqref{crosscut} with $C=P_n$).
Similarly, by considering the dual order on $L_n$ we obtain
\begin{gather}
\label{eq:19}
\sum_{\substack{A\subseteq L_n\setminus\{1,n\} \\ \lcm(A)=n}} (-1)^{|A|} =
\sum_{\substack{A\subseteq P_n \\ \lcm(A)=n}} (-1)^{|A|} =
\sum_{\substack{A\subseteq P_n\\ \gcd(A^\ast)=1}} (-1)^{|A|} =
\sum_{\substack{A\subseteq P_n^\ast \\ \gcd(A)=1}} (-1)^{|A|} = \mu(n)
\end{gather}
for each non-prime integer $n\ge 1$. For $n>1$, the requirement that $n$ is
non-prime is necessary in order to ensure that $P_n,P_n^\ast\subseteq
L_n\setminus\{1,n\}$ in Eqs.~\eqref{eq:12} and \eqref{eq:19}.  It can be
omitted by considering $L_n\setminus\{n\}$ in Eq.~\eqref{eq:12} and
$L_n\setminus\{1\}$ in Eq.~\eqref{eq:19}, respectively.

If $n$ is not squarefree, then $\mu(n)=0$ and hence, due to Eqs.~\eqref{eq:12}
and \eqref{eq:19}, the abstract simplicial complexes
\begin{align*}
\mathscr{S}_n & = \{A\subseteq L_n\setminus\{1,n\}\mathrel| \gcd(A)>1\}, \\
\mathscr{T}_n & = \{A\subseteq L_n\setminus\{1,n\}\mathrel|
\text{$A\neq\emptyset$ and $\lcm(A)<n$}\}
\end{align*}
have Euler characteristic~1.  Recall that the Euler characteristic
$\chi(\mathscr{A})$ of an abstract simplicial complex $\mathscr{A}$ is defined
as $\chi(\mathscr{A}) = \sum_{A\in\mathscr{A}} (-1)^{|A|-1}$, and that
$\mathscr{A}$ is called \emph{contractible} if its geometric realization as a
simplicial complex is contractible, which means, roughly speaking, that it can
be continuously shrunk to a point.  It is well-known that if $\mathscr{A}$ is
contractible, then $\chi(\mathscr{A})=1$.  In view of this, one might
conjecture that both $\mathscr{S}_n$ and $\mathscr{T}_n$ are contractible if
$n$ is not squarefree. This is indeed the case: Suppose $p^2\mathrel|n$ for
some prime~$p$.  For any $A\in\mathscr{S}_n$, if
$\gcd(A)\mathrel|\frac{n}{p}$, then $\gcd(A\cup\{\frac{n}{p}\}) = \gcd(A) >
1$; if $\gcd(A)\nmid \frac{n}{p}$, then $p^2\mathrel| \gcd(A)$ and hence,
$\gcd(A\cup\{\frac{n}{p}\})\ge p > 1$.  In both cases, $A\cup\{\frac{n}{p}\}
\in\mathscr{S}_n$. Thus, $\frac{n}{p}$ is contained in every maximal face of
$\mathscr{S}_n$. As a consequence, the geometric realization of
$\mathscr{S}_n$ is star-shaped with respect to $\frac{n}{p}$ and hence
contractible.  Similarly, by distinguishing the cases $p\mathrel|\lcm(A)$ and
$p\nmid\lcm(A)$ we may conclude that $\mathscr{T}_n$ is contractible.  In
fact, $\mathscr{S}_n$ and $\mathscr{T}_n$ are isomorphic by virtue of
$A\mapsto A^\ast$.

The preceding contractibility result establishes a link to the theory of
discrete tubes. Due to Corollary~2 of \cite{NW:1997:AT}, for any contractible
abstract simplicial complex $\mathscr{A}$,
\begin{equation}
(-1)^r \sum_{A\in\mathscr{A} \atop |A|\le r} (-1)^{|A|-1} \le (-1)^r \quad
(r=1,2,3,\dots).
\label{eq:21}
\end{equation}
Applying this to $\mathscr{A}=\mathscr{S}_n$ resp.\ $\mathscr{T}_n$ gives
Bonferroni-like inequalities on $\gcd$- and $\lcm$-sums, e.g., in
Eqs.~\eqref{eq:12} and \eqref{eq:19}, in the particular case where $n$ is
non-squarefree.

\subsubsection{Euler's totient function}
\label{sec:eulers-toti-funct-1}

Our conclusions on Euler's totient function and its Dirichlet inverse
(cf.~Subsection \ref{sec:dirichl-inverse-eule}) are stated more generally
using the notion of a multiplicative function.  We refer to any function
$h:\mathbb{N}\rightarrow \mathbb{C}$ as \emph{multiplicative} if $h(1)=1$ and
$h(ab) = h(a)h(b)$ for any coprime $a,b\in\mathbb{N}$, and as \emph{completely
  multiplicative} if the latter condition holds for any $a,b\in\mathbb{N}$.

Examples of multiplicative functions are the identity function, the power
functions for any complex exponent, and the Liouville function, which are all
completely multiplicative; further examples include the M\"obius function,
Euler's totient function, and the sum of positive divisors of $n$.  

Let $f(A)=(-1)^{|A|-1}h(\gcd(A))$ for $A\subseteq L_n\setminus\{1,n\}$ where
$h$ is multiplicative and non-vanishing on the set of primes.  Then, $h(n)\neq
0$ for $n\in\mathbb{N}$---provided $n$ is squarefree or $h$ is completely
multiplicative. In both cases, $h(d)\neq 0$ and $h(n/d) = h(n)/h(d)$ for any
positive divisor $d$ of $n$.  In view of this, the first remark following
Theorem \ref{thm:1} implies that for any non-prime integer $n\ge 1$,
\begin{multline}
\sum_{\substack{A\subseteq L_n\setminus\{1,n\}\\ A\neq\emptyset}} (-1)^{|A|-1}
h(\gcd(A)) = \sum_{\substack{A\subseteq P_n^\ast\\ A\neq\emptyset}}
(-1)^{|A|-1} h(\gcd(A)) = \sum_{\substack{A\subseteq P_n \\ A\neq\emptyset}}
(-1)^{|A|-1} h(\gcd(A^\ast))
\\
= \sum_{\substack{A\subseteq P_n \\ A\neq\emptyset}} (-1)^{|A|-1} h\left( n
\prod_{a\in A} \frac{1}{a} \right) = h(n) - h(n) \sum_{A\subseteq P_n}
(-1)^{|A|} \prod_{a\in A} \frac{1}{h(a)}
\label{eq:totient}
\end{multline}
and hence,
\begin{align}
\label{eq:33}
\sum_{\substack{A\subseteq L_n\setminus\{1,n\}\\ A\neq\emptyset}} (-1)^{|A|-1}
h(\gcd(A)) & = h(n) - h(n) \prod_{\substack{p|n \\ \text{$p$ prime}}} \left( 1
- \frac{1}{h(p)} \right)  =: h(n) - \varphi_h(n)
\end{align}
provided $n$ is squarefree or $h$ is completely multiplicative.  For
$h=\text{id}_{\mathbb{N}}$ the function $\varphi_h$ in Eq.~\eqref{eq:33} is
known as Euler's totient function, which for any $n\in\mathbb{N}$ evaluates to
the number of positive integers coprime with $n$ (sequence A00010 in
\cite{OEIS:2014}).

In a similar way to Subsection \ref{sec:class-mobi-funct}, the requirement
that $n$ is non-prime can be dropped by considering $L_n\setminus\{n\}$
instead of $L_n\setminus\{1,n\}$ in Eqs.~\eqref{eq:totient} and \eqref{eq:33}.
Furthermore, if $n$ is not squarefree and $h$ completely multiplicative, then
$\mu(n)=0$ and hence by Eq.~\eqref{eq:12} the sum in Eq.~\eqref{eq:33} (even
with the previous modification) can be restricted to $\gcd(A)>1$.

Since $L_n\setminus\{n\}$ is a lower semilattice for any $n\in\mathbb{N}$, we
obtain by Eqs.~\eqref{eq:11} and \eqref{eq:33},
\begin{align}
\label{eq:13}
h(n)-\varphi_h(n) & = \sum_{\substack{A\subseteq L_n\setminus\{n\}\\ \text{$A$
      chain} \\ A\neq\emptyset}} (-1)^{|A|-1} h(\gcd(A))
= \sum_{d\in L_n\setminus\{n\}} h(d)\! \sum_{\substack{A\subseteq L_n\setminus\{n\}\\ \text{$A$ chain}\\
    \gcd(A)=d}} (-1)^{|A|-1} .
\end{align}
By backward induction on the height of $d$ in $L_n\setminus\{n\}$ it can be
shown that the inner sum in Eq.~\eqref{eq:13} agrees with $-\mu(n/d)$. As a
consequence, 
\begin{align}
\label{eq:14}
\varphi_h(n) & = \sum_{d|n} h(d) \mu\left( \frac{n}{d} \right) = h(n) \sum_{d|n}
\frac{\mu(d)}{h(d)}
\end{align}
provided $n$ is squarefree or $h$ is completely multiplicative. 
Under this requirement, we rediscover the known formula (cf.~Subsection \ref{sec:dirichl-inverse-eule})
\begin{gather}
\label{eq:29}
\prod_{\substack{p|n \\ \text{$p$ prime}}} \left( 1 - \frac{1}{h(p)} \right) =
\sum_{d|n} \frac{\mu(d)}{h(d)} 
\end{gather}
as an immediate consequence of Eq.~\eqref{eq:14}.  Eq.~\eqref{eq:29} also
holds for non-squarefree numbers~$n$ and any multiplicative function $h$ if
$h$ is required to be nowhere zero, or if the sum in Eq.~\eqref{eq:29} is
restricted to $d|n$ where $d$ is squarefree. Both modifications immediately
follow by applying Eq.~\eqref{eq:29} to the squarefree kernel of $n$.

\subsubsection{Dirichlet inverse of Euler's totient function}
\label{sec:dirichl-inverse-eule}

The dual of Eq.~\eqref{eq:10}, applied to $f(A)=(-1)^{|A|}h(\lcm(A))$ for any
$A\subseteq L_n\setminus\{1,n\}$ where $h$ is multiplicative reveals that for
any non-prime integer $n\ge 1$,
\begin{gather}
\label{eq:15}
\sum_{A\subseteq L_n\setminus\{1,n\}} \!(-1)^{|A|} h(\lcm(A))
= \sum_{A\subseteq P_n} (-1)^{|A|} \prod_{a\in A} h(a) 
= \! \prod_{\substack{p|n \\ \text{$p$ prime}}} \! (1-h(p)) .
\end{gather}
For $h=\text{id}_{\mathbb{N}}$ the product on the right-hand side of
Eq.~\eqref{eq:15} is known as the \emph{Dirichlet inverse} of Euler's totient
function (sequence A023900 in \cite{OEIS:2014}).  Similar to our discussion on
the totient function, the requirement that $n$ is non-prime can be removed by
considering $L_n\setminus\{1\}$ instead of $L_n\setminus\{1,n\}$ in
Eq.~\eqref{eq:15}.  Furthermore, if $n$ is not squarefree, then $\mu(n)=0$ and
hence by Eq.~\eqref{eq:19}, the sum in Eq.~\eqref{eq:15} can be restricted to
$\lcm(A)<n$.  Since $L_n\setminus\{1\}$ is an upper semilattice for any
$n\in\mathbb{N}$, we obtain by Eqs.~\eqref{eq:11} and \eqref{eq:15},
\begin{align}
\label{eq:16}
\prod_{\substack{p|n\\ \text{$p$ prime}}} (1-h(p)) = \sum_{\substack{A\subseteq L_n\setminus\{1\} \\
    \text{$A$ chain}}} (-1)^{|A|} h(\lcm(A)) = 1 + \sum_{d\in
  L_n\setminus\{1\}} h(d) \!  \sum_{\substack{A\subseteq L_n\setminus\{1\} \\
    \text{$A$ chain} \\ \lcm(A)=d}} (-1)^{|A|} .
\end{align}
By induction on the height of $d$ in $L_n\setminus\{1\}$ it follows that the
inner sum in Eq.~\eqref{eq:16} agrees with $\mu(d)$. Thus, the following known
formula (cf.~Theorem 2.18 in \cite{Apostol:1976}) is obtained:
\begin{align}
\label{eq:17}
\prod_{\substack{p|n\\ \text{$p$ prime}}} (1-h(p)) = \sum_{d|n} h(d) \mu(d) .
\end{align}
Note that in Eqs.~\eqref{eq:15}--~\eqref{eq:17} we do not impose any further
requirement on $h$.  By applying Eq.~\eqref{eq:17} to $1/h$ where $h$ is
multiplicative and nowhere zero, we rediscover Eq.~\eqref{eq:29}.

\subsubsection{Riemann zeta function}

Closely related to $\varphi_h$ in Eq.~\eqref{eq:33} is the $\zeta$-function,
which can be represented as
\begin{align}
\label{eq:20}
\frac{1}{\zeta(s)} & = \lim_{n\rightarrow\infty} \frac{\varphi_h(n!)}{h(n!)},
\quad \re(s)>1,
\end{align}
where $h(n)=n^s$ for any $n\in\mathbb{N}$.  By
Eq.~\eqref{eq:33},
\begin{align}
\label{eq:35}
\frac{1}{\zeta(s)} & = 1 + \lim_{n\rightarrow \infty} \,\frac{1}{(n!)^s}
\! \sum_{A\subseteq L_{n!}\setminus\{1,n!\}} (-1)^{|A|}
\left( \gcd(A) \right)^s, \quad \re(s)>1.
\end{align}
In particular, for $s=2$,
\begin{align}
\label{eq:36}
\lim_{n\rightarrow \infty} \,\frac{1}{(n!)^2}
\! \sum_{A\subseteq L_{n!}\setminus\{1,n!\}} (-1)^{|A|-1}
\left( \gcd(A) \right)^2 = 1 - \frac{6}{\pi^2} .
\end{align}
Eqs.~\eqref{eq:20}--\eqref{eq:36} also hold if $n!$ is replaced by $n\#$
where $n\#$ denotes the primorial of $n$, that is, the product of all primes
less than or equal to $n$ (sequence A034386 in \cite{OEIS:2014}).

\section{Generalization to convex geometries}
\label{sec:furth-gener}

A \emph{closure system} $(S,h)$ consists of a set $S$ and a hull operator $h$
on $S$, i.e.\ an extensive, increasing, idempotent operator on subsets of~$S$.
A subset $A$ of $S$ is called \emph{$h$-closed} if $h(A)=A$, and
\emph{$h$-free} if all subsets of $A$ are $h$-closed.  An \emph{$h$-basis} of
$A$ is a minimal subset $B$ of $A$ such that $h(B)=A$.  A \emph{convex
  geometry} is a closure system $(S,h)$ where $S$ is finite and any $h$-closed
subset of $S$ has a unique $h$-basis \cite{EJ:1985:TCG}.

\begin{theorem}
\label{FurtherGen}
Let $(S,h)$ be a convex geometry, $\Gamma$ an abelian group (additively
written), and $f:2^S \rightarrow \Gamma$ such that for any $h$-closed, but not
$h$-free subset $A$ of $S$,
\begin{gather}
\label{eq:1}
\sum_{I:\,A_0\subseteq I\subseteq A} f(I) = 0\, ,
\end{gather}
where $A_0$ denotes the unique $h$-basis of $A$.  Then,
\begin{gather}
\label{eq:2}
\sum_{A\subseteq S} f(A) = \sum_{\substack{A\subseteq S\\ \text{$A$ $h$-free}}}
f(A) \, .
\end{gather}
\end{theorem}

\begin{proof}
Since $(S,h)$ is a convex geometry, $h(I)=A$ if and only if $A_0\subseteq
I\subseteq A$. Hence,
\begin{gather*}
 \sum_{A\subseteq S} f(A) = \sum_{\substack{A\subseteq S \\ \text{$A$
      $h$-closed}}} \sum_{I:\,h(I)=A} f(I) = \sum_{\substack{A\subseteq S \\
    \text{$A$
      $h$-closed}}} \sum_{I:\,A_0\subseteq I\subseteq A} f(I)= \sum_{\substack{A\subseteq S \\
    \text{$A$ $h$-free}}} \sum_{I:\,A_0\subseteq I\subseteq A} f(I) \, .
\end{gather*}
Since any $A\subseteq S$ is $h$-free if and only if $A_0=A$,
the result follows.
\end{proof}

\begin{remark}
The requirement in Eq.~(\ref{eq:1}) is satisfied if $f(I) = (-1)^{|I|}
\gamma(h(I))$ for any $I\subseteq S$ where $\gamma : 2^S\rightarrow \Gamma$.
In this case, we obtain
\begin{gather}
\label{eq:3}
\sum_{A\subseteq S} (-1)^{|A|} \gamma(h(A)) = \sum_{\substack{A\subseteq S\\
    \text{$A$ $h$-free}}} (-1)^{|A|} \gamma(A) \, .
\end{gather}
In particular, by defining $\gamma(A) = (-1)^{|A|}$ for any $A\subseteq S$,
the sum $\sum_{A\subseteq S} (-1)^{|h(A)|-|A|}$ on the left-hand side of
Eq.~(\ref{eq:3}) evaluates to the number of $h$-free subsets of $S$, while by
defining $\gamma(A)=1$ for any $A\subseteq S$, Eq.~(\ref{eq:3}) reveals that
the Euler characteristic of the abstract simplicial complex of all non-empty
$h$-free subsets of $S$ is equal to 1, provided $S\neq\emptyset$. This latter result is attributed to Lawrence (unpublished,
cf.~\cite{EJ:1985:TCG}).
\end{remark}

\begin{remark}
The preceding theorem can be generalized even further by requiring that
$(S,h)$ is a closure system and replacing Eq.~(\ref{eq:1}) by
\begin{gather*}
\sum_{\emptyset \neq \mathscr{J} \subseteq \mathscr{A}_0}
(-1)^{|\mathscr{J}|-1} \! \sum_{I:\,\bigcup\!\!\mathscr{J}\subseteq I\subseteq
  A} \! f(I)
=0
\end{gather*}
where $\mathscr{A}_0$ denotes the system of all $h$-bases of $A$. Note that in
this more general setting, $A$ is $h$-free if and only if $\mathscr{A}_0 =
\{A\}$.
\end{remark}

In the following, we derive Theorem \ref{thm:1} from Theorem \ref{FurtherGen}.

\begin{proof}[Proof of Theorem~\ref{thm:1}]
The requirements imply that for any $B\in\mathscr{B}$ there is some $c(B)\in S
\setminus B$ such that $B\cup \{c(B)\} \in \mathscr{C}$ and $c(B)>b$ for any
$b\in B$.  For any $A\subseteq S$ define
\begin{align*}
\mathscr{B}|_A & :=\{B\in\mathscr{B}\mathrel| B\subseteq A\}, \\
h(A) & := A \cup \{ c(B)\mathrel| B\in\mathscr{B}|_A\}, \\
h^\ast(A) & := h(A) \cup h(h(A)) \cup \dots
\end{align*}
Then, $h^\ast$ is a hull operator on $S$, and 
\begin{align*}
A_0 & := A \setminus \{ c(B) \mathrel| B\in\mathscr{B}|_A\} 
\end{align*}
is the unique $h^\ast$-basis of any $h^\ast$-closed subset $A$ of $S$.

In order to verify Eq.~(\ref{eq:1}), let $A\subseteq S$ be $h^\ast$-closed,
but not $h^\ast$-free.  Then, $A_0\neq A$ and hence,
$\mathscr{B}|_A\neq\emptyset$.  Choose $B'\in\mathscr{B}|_A$ such that $c(B')
= \min \{c(B)\mathrel| B\in\mathscr{B}|_A\}$.  Since $B'\subseteq A$ and $A$
is $h^\ast$-closed, $c(B')\in A$ and therefore, $B'\cup \{c(B')\} \subseteq
A$.  We observe that $B'\subseteq A_0$, since otherwise $B' \cap
\{c(B)\mathrel| B\in\mathscr{B}|_A\}\neq \emptyset$, which implies $c(B)\le
\max B' < c(B')$ for some $B\in\mathscr{B}|_A$, contradicting the minimality
of $c(B')$. Now,
\begin{align*}
\sum_{I:A_0\subseteq I\subseteq A} f(I) & = \sum_{\substack{I:A_0\subseteq I
    \subseteq A \\ \max C\in I}} f(I) + \sum_{\substack{I:A_0\subseteq I
    \subseteq A \\ \max C\notin I}} f(I) 
= \sum_{\substack{I:A_0\subseteq I \subseteq A \\ \max C\in I}} (f(I) +
f(I\setminus \{\max C\})) 
\end{align*}
where $C:=B'\cup \{c(B')\}$.  Since any $I$ in the latter sum includes $C$,
Eq.~(\ref{Bedingung}) (with $I$ in place of $A$) reveals $f(I) + f(I\setminus
\{\max C\})=0$; hence, the whole sum vanishes as
required in Eq.~(\ref{eq:1}).  Applying Theorem \ref{FurtherGen} now gives a
sum over all $h^\ast$-free subsets of $S$.  Since any $A\subseteq S$ is
$h^\ast$-free if and only if $B\not\subseteq A$ for any $B\in\mathscr{B}$, the
proof is complete.
\end{proof}

%%%%%%%%%%%%%%%%%%%%%%%%%%%%%%%%%%%%%%%%%%%%%%%%%%%%%%%

\bibliography{main}{}
\bibliographystyle{plain}
%\printbibliography[title={References}, ]

\end{document}